\documentclass[11pt,reqno,draft]{amsart}
\usepackage{latexsym,amssymb}
\usepackage[all]{xy}

\newtheorem{defi}{Definition}[section]
\newtheorem{proposition}[defi]{Proposition}
\newtheorem{cor}[defi]{Corollary}

\newtheorem{lemma}[defi]{Lemma}
\newtheorem{theorem}[defi]{Theorem}

\newcommand{\defeq}{\mathrel{\mathrm{\raise0.1ex\hbox{:}\hbox{=}\strut}}}
 \DeclareMathOperator{\Tr}{Tr}
 
 \DeclareMathOperator{\Lip}{Lip}

\def\R{\mathbb R}
\def\N{\mathbb{N}}
\def\E{\mathbb E}

\def\la{\langle}
\def\ra{\rangle}

\def\ve{\varepsilon}

\def\test{\mathcal{F}C_b^{2}(D(A))}

\def\V{V_{\frac 12}}
\def\V-{V_{-\frac 12}}

\def\cB{\mathcal{B}}

\begin{document}
\numberwithin{equation}{section}

\title[Improved moment estimates for invariant measures of SPDE's]
{Improved moment estimates for invariant measures of semilinear diffusions
in Hilbert spaces and applications}
\author[A. Es--Sarhir]{Abdelhadi Es--Sarhir}
\author[W. Stannat]{Wilhelm Stannat}
\address{Technische Universit\"at Berlin, Fakult\"at II, Institut f\"ur Mathematik,
Sekr. Ma 7-4\newline Stra{\ss}e des 17. Juni 136, D-10623 Berlin, Germany}
\email{essarhir@math.tu-berlin.de}
\address{Technische Universit\"at Darmstadt, Fachbereich Mathematik, \newline Schlo\ss gartenstra\ss e 7,
D-64289 Darmstadt, Germany}
\email{stannat@mathematik.tu-darmstadt.de}



\begin{abstract}
We study regularity properties for invariant measures of semilinear
diffusions in a separable Hilbert space. Based on a pathwise
estimate for the underlying stochastic convolution, we prove a
priori estimates on such invariant measures. As an application, we
combine such estimates with a new technique to prove the
$L^1$-uniqueness of the induced Kolmogorov operator, defined on a
space of cylindrical functions. Finally, examples of stochastic
Burgers equations and thin-film growth models are given to
illustrate our abstract result.
\end{abstract}

\maketitle

\section{Introduction}
\noindent
The aim of this work is to obtain improved moment estimates of invariant measures
of semilinear stochastic evolution equations of the type
\begin{equation}
\label{sde0}
dX(t)=\Big(AX(t)+B(X(t))\Big)dt + \sqrt{Q}dW_t,\quad t\geq 0
\end{equation}
defined on a separable real Hilbert space $H$. Here $A$ is a self-adjoint linear
operator of negative type $\omega$ on $H$ having a compact resolvent, $B$ is a
nonlinear function with subdomain $D(B)\subset H$. $Q$ is a symmetric positive
definite operator and $(W_t)_{t\ge 0}$ is a cylindrical Wiener process in $H$
defined on a filtered probability space $(\Omega,\mathcal{F},(\mathcal{F}_t)_{t\ge 0},
\mathbb{P})$.

\medskip
\noindent Equation \eqref{sde0} can be read as an abstract
formulation of many partial differential equations perturbed by
random noise such as stochastic reaction diffusion, Allen-Cahn,
Burgers and Navier-Stokes equations. Existence and uniqueness of
solutions to such equations are well studied, we refer to the
monographs by Da Prato, Zabczyk \cite{DaZ1,DaZ2}, Cerrai
\cite{cerrai01} and the works \cite{Da-De-Te,GoMa}. We will be in
particular interested in the situation, where \eqref{sde0} has a
mild solution $X(t)$, $t\ge 0$, with a time-invariant distribution
$\mu = \mathbb{P}\circ X(t)^{-1}$. Throughout this paper, we call
such a solution a stationary mild solution and $\mu$ an invariant
measure of \eqref{sde0}. Given such a stationary mild solution, we
will then derive in Section \ref{section3} moment estimates on its
time-invariant distribution $\mu$ under appropriate assumptions on
the coefficients of \eqref{sde0}.

\medskip
\noindent
Moment estimates for invariant measures of stochastic partial differential equations have been
studied quite intensively for some time. Recently, in the case where $B$ is locally Lipschitz, the
authors proved in \cite{Es-St} existence and moment estimates of an invariant measure $\mu$
corresponding to \eqref{sde0} under a Lyapunov type assumption on the coefficients $A$ and $B$.
These moment estimates have been the main tool to discuss well-posedness of the parabolic Cauchy
problem corresponding to stochastic reaction diffusion or Allen-Cahn equations in $L^1(\mu)$.
However, there are many important examples, e.g. the stochastic Burgers equation, that are still
not covered by our analysis. The results in this paper can be seen as improved moment
estimates on invariant measures to semilinear diffusions under weaker assumptions on its coefficients.

\medskip
\noindent
The main ingredient, to obtain our moment estimates, is a
pathwise control on the stochastic convolution arising in the mild
formulation of \eqref{sde0}. This idea is taken from the paper
\cite{Fl-Ga:94} by Flandoli and Gatarek on stochastic Navier-Stokes
equations, see also the paper \cite{Da-De:07} by Da Prato and
Debussche where the same idea has been applied to the stochastic
Burgers equation. We have generalized this technique and found
simplified proofs to apply the same technique in an abstract
context. To illustrate this result we discussed at the end examples
of stochastic Burgers equations and thin-film growth models. We
shall remark that the same result can be proved for stationary
solutions of stochastic Navier-Stokes equations in the spirit of
Flandoli and Gatarek \cite{Fl-Ga:94}.

\medskip
\noindent
The existence of a stationary mild solution is a rather
weak assumption on the equation \eqref{sde0} and in particular does
not imply neither the existence of an associated full Markov process
nor an associated transition semigroup $(P_t)_{t\ge 0}$. The
existence of $(P_t)_{t\ge 0}$, however, can be obtained from the
Hille-Yosida theory, in the case, where the Kolmogorov operator
associated with \eqref{sde0} $(L,D(L))$ (resp. its closure on
suitable test functions) generates a $C_0$-semigroup in $L^1 (H, \mu
)$. Based on the improved moment estimates on $\mu$ we will
therefore study the existence (and uniqueness) of $(P_t)_{t\ge 0}$
in Section 4. The method which we follow here is new and different
to the one presented in \cite{St:99} due the fact that the drift
term $B$ is not supposed to be dissipative and the coefficients of
the finite dimensional realization of $L$ are not bounded. Hence we
can not use the classical theory by \cite{krylov} to obtain
uniform gradient estimates for the pseudo-resolvents associated with
finite dimensional approximations of $L$.

\bigskip
\noindent
Let us now specify our precise assumptions:

\begin{enumerate}

\item[${\bf (H_0)}$] $A$ is selfadjoint, $\|e^{tA}\|\leq e^{-\omega t}$ for certain
$\omega>0$ and its resolvent $A^{-1}$ (which exists) is compact.

\item[${\bf (H_1)}$] $B : D(B)\subset H\to H$ is a measurable vector-field, defined on a measurable
subset $D(B)\subset H$. We will always consider $B$ as everwhere defined, by setting $B(h) = 0$ for
$h\notin D(B)$.

\item[${\bf (H_2)}$] $Q$ is a bounded, nonnegative, symmetric operator such that $A$ and $Q$ are
simultaneously diagonizable and there exist $\nu\in ]0, \frac 12[$ such that for all $t > 0$
$$
\int_0^t s^{-2\nu} \|\sqrt{Q}e^{sA}\|^2_{HS} \, ds < \infty\, .
$$
\item [${\bf (H_3)}$] There exists a mild solution
\begin{equation*}
X(t)=e^{tA}X_0 + \int_0^t e^{(t-s)A}B(X(s))ds+\int_0^t e^{(t-s)A}\sqrt{Q}dW_s,\quad t\ge 0\, ,
\end{equation*}
of \eqref{sde0}  having a time-invariant distribution $\mu =
\mathbb{P}\circ X(t)^{-1}$.
\end{enumerate}

%in the space $C([0,T],H)\cap L^2([0,T],V_{\frac 12})$ for all
%$T\ge0$,

\medskip
\noindent We shall introduce the following interpolation spaces: For
$\theta\in\R$ let
$$
V_\theta\defeq (D((-A)^{\theta}),\|\cdot\|_{\theta}),\quad\mbox{ where $\|x\|_{\theta}
= \la (-A)^{\theta} x,(-A)^{\theta} x\ra$ for $x\in V_{\theta}$}.
$$

\noindent
Hypotheses ${\bf (H_2)}$ implies that the stochastic convolution $W_A(t)$ defined by
$$
W_A(t)\defeq \int_0^t e^{(t-s)A}\sqrt{Q}dW_s
$$
is well defined and satisfies the uniform moment estimate
\begin{equation}\label{M}
M := \sup\limits_{t\geq0}\E\left( \|W_A(t)\|_{\gamma}^2 \right)
   = \int_0^{\infty}\|(-A)^{\gamma}e^{tA}\sqrt{Q}\|_{HS}^2\:dt <
\infty,\quad 0 < \gamma < \nu\, .
\end{equation}

See \cite {Es,Es-St:09} for more details.

\section{Pathwise estimates for stochastic convolutions}
\label{section2}

\noindent
The aim of this section is to prove a pathwise estimate for the stochastic convolution
associated with the linear operator $A$. The estimate will be useful in the next section
to obtain improved moment estimates on $\mu$. We start with the following 1-dimensional result:

\begin{proposition}
\label{1-dim-est}
Let $(\beta(t))_{t\geq 0}$ be a $1$-dimensional Brownian motion. For $t\geq 0$  set
\begin{equation}
W_{-\lambda}(t)=\int_0^te^{-\lambda(t-s)}\:d\beta(s),\quad
\lambda>0.
\end{equation}
Then for all $\delta\in(0,\frac 12)$
\begin{equation}
\sup\limits_{0\leq t\leq T}|W_{-\lambda}(t)|\leq
\lambda^{-\delta}\cdot C_{\delta} M(\delta, T)
\end{equation}
with
$$
C_{\delta}\defeq
\Gamma(\delta+1)+\delta^{\delta}e^{-\delta},\quad M(\delta, T)\defeq
\sup\limits_{0\leq s\leq t\leq T}\frac{|\beta(t)-\beta(s)|}{|t-s|^{\delta}}.
$$

\smallskip
\noindent
Moreover,
\begin{equation}
\label{est} \E(M(\delta, T)^m)\leq M\cdot T^{m(\frac
12-\delta)}\quad\mbox{ for all }\: m\geq 1.
\end{equation}
for some constant $M$ that is independent of $\lambda$ and $T$.
\end{proposition}

\begin{proof}
It\^{o}'s product rule implies that
\begin{equation}
\begin{split}
W_{-\lambda}(t)&=\beta(t)-\lambda\int_0^te^{-\lambda(t-s)}\beta(s)\:ds\\
&=\lambda
\int_0^te^{-\lambda(t-s)}(\beta(t)-\beta(s))\:ds+e^{-\lambda
t}\beta(t),
\end{split}
\end{equation}
so that for $t\leq T$
\begin{equation}
\begin{split}
|W_{-\lambda}(t)|&\leq \lambda
\int_0^te^{-\lambda(t-s)}(t-s)^{\delta}\:ds\cdot M(\delta,
T)+e^{-\lambda t}\cdot t^{\delta}\cdot M(\delta, T)\\
&\leq \Big(\lambda \int_0^{+\infty} e^{-\lambda
s}s^{\delta}\:ds+\delta^{\delta}e^{-\delta}\cdot
\lambda^{-\delta}\Big)\cdot M(\delta,T)\\
&=\lambda^{-\delta}\Big(\Gamma(\delta+1)+\delta^{\delta}e^{-\delta}\Big)\cdot
M(\delta, T).
\end{split}
\end{equation}
The moment estimate \eqref{est} follows from Th\'eor\`eme 3 in \cite{R} (see also \cite{BY}).
\end{proof}

\noindent
We can now apply the Proposition to obtain a pathwise estimate on the stochastic convolution
$$
W_{A-\lambda}(t)\defeq \int_0^te^{(t-s)(A-\lambda)}\sqrt{Q}\:dW(s),\quad \lambda>0.
$$
To this end, denote by $(\lambda_{k})_{k\geq 1}$ and $(q_{k})_{k\geq 1}$ the eigenvalues of
$-A$ and $Q$ respectively corresponding to the same eigenbasis $(e_k)_{k\geq 1}$ in $H$.
Then the last Proposition implies

\begin{cor}
\label{Cor}
Let $\delta\in (0,\frac 12)$ and $\gamma\in \R$. Then
\begin{equation*}
\sup\limits_{0\leq t\leq T}\|W_{A-\lambda}(t)\|_{{\gamma}}^2\leq
C_{\delta}^2\sum\limits_{k=1}^{+\infty}\frac{\lambda_k^{2\gamma}q_k}{(\lambda+\lambda_k)^{2\delta}}
M_k(\delta, T)^2.
\end{equation*}
Here,
$$
M_k(\delta, T) \defeq \sup\limits_{0\leq s<t\leq T}\frac{|\beta_k(t)-\beta_k(s)|}{|t-s|^{\delta}},\quad k\geq 1
$$
are independent random variables satisfying the moment estimate \eqref{est}. In particular, if there exists
$\ve>0$ such that
\begin{equation}
\label{Z-est}
Z_{\gamma,\delta,\ve}\defeq \sum\limits_{k\geq 1}^{+\infty}\lambda_k^{-2(\delta-\gamma-\ve)}q_k<+\infty,
\end{equation}
then
\begin{equation}\label{W-est}
\sup\limits_{0\leq t\leq T}\|W_{A-\lambda}(t)\|_{{\gamma}}^2\leq
\lambda^{-2\ve}\cdot M_{\delta,\gamma,\ve}
\end{equation}
for some random variable $M_{\delta,\gamma,\ve}$, independent of $\lambda$, having finite moments of any order.
\end{cor}

\begin{proof}
Clearly,
\begin{equation*}
\|W_{A-\lambda}(t)\|^2_{\gamma}=\sum\limits_{k=1}^{+\infty}\lambda_k^{2\gamma}\la
W_{A-\lambda}(t),e_k\ra^2=\sum\limits_{k=1}^{+\infty}\lambda_k^{2\gamma}
\left(\int_0^te^{-(\lambda+\lambda_k)(t-s)}\sqrt{q_k}\:d\beta_{k}(s)\right)^2,
\end{equation*}
where $\beta_k$, $k\geq 1$, are independent $1$-dimensional Brownian motions. Proposition \ref{1-dim-est}
now implies that
\begin{equation*}
\begin{split}
\sup\limits_{0\leq t\leq
T}\|W_{A-\lambda}(t)\|_{{\gamma}}^2&\leq\sum\limits_{k=1}^{+\infty}\lambda_k^{2\gamma}q_k\sup\limits_{0\leq
t\leq T}|W_{-(\lambda+\lambda_k)}^{(k)}(t)|^2
\leq C_{\delta}^2
\sum\limits_{k=1}^{+\infty}\frac{\lambda_k^{2\gamma}q_k}{(\lambda+\lambda_k)^{2\delta}}\cdot
M_k(\delta, T)^2.
\end{split}
\end{equation*}
If $Z_{\gamma,\delta,\ve}<+\infty$, then
\begin{equation}
\label{est-stoch-conv}
\sup\limits_{0\leq t\leq T}\|W_{A-\lambda}(t)\|_{{\gamma}}^2\leq
\lambda^{-2\ve}\cdot M_{\delta,\gamma,\ve}
\end{equation}
with
$$
M_{\delta,\gamma,\ve}\defeq C_\delta^2\, \sum\limits_{k=1}^{+\infty}
\lambda_k^{-2(\delta-\gamma-\ve)}q_k M_k(\delta, T)^2.
$$

\medskip
\noindent
For the proof of the last statement of the corollary, take $m\geq 1$. By Jensen's inequality we can write
\begin{equation*}
M_{\delta,\gamma,\ve}^m = Z_{\gamma,\delta,\ve}^m\left(\frac{1}{Z_{\gamma,\delta,\ve}}
\sum\limits_{k=1}^{+\infty} \lambda_k^{-2(\delta-\gamma-\ve)}q_k
M_k(\delta, T)^2\right)^m
\leq Z_{\gamma,\delta,\ve}^{m-1}\sum\limits_{k=1}^{+\infty}
\lambda_k^{-2(\delta-\gamma-\ve)}q_k M_k(\delta, T)^{2m}
\end{equation*}
and using the moment estimate \eqref{est} we conclude that
\begin{equation}
\label{r-moment of M}
\begin{split}
\E \Big(M_{\delta,\gamma,\ve}\Big)^m
& \leq Z_{\gamma,\delta,\ve}^{m-1}\sum\limits_{k=1}^{+\infty}\lambda_k^{-2(\delta-\gamma-\ve)}q_k M
   \cdot T^{m(1-2\delta)}
 = M\cdot Z_{\gamma,\delta,\ve}^{m-1}\cdot T^{m(1-2\delta)}<\infty,
\end{split}
\end{equation}
where $M$ is a universal constant.
\end{proof}

%%%%%%%%%%%%%%%%%%%%%%%%%%%%%%%%%%%%%%%% Section 3 %%%%%%%%%%%%%%%%%%%%%%%%%%%%%%%%%%%%%%%%%%%

\section{ A priori estimates on invariant measures}
\label{section3}

\noindent
In this section we will prove improved moment estimates on
the invariant distribution $\mu$ of a stationary mild solution of
\eqref{sde0}. The existence of a stationary mild solution is known
in many important applications that are covered by our setting,
especially for stochastic Burgers equations and thin-film growth
models (see Section \ref{SectionExamples} below). For our analysis
we need the following assumptions. Fix $0\leq \gamma_1\leq \gamma_2$
and assume

\begin{enumerate}
\item[${\bf (H_4)}$] There exists $\ve>0$ such that
\begin{equation*}
Z_{\gamma_2,\delta,\ve}\defeq \sum\limits_{k\geq
1}^{+\infty}\lambda_k^{-2(\delta-\gamma_2-\ve)}q_k < +\infty,
\end{equation*}

\item[${\bf (H_5)}$] There exist positive constants $\alpha$, $\beta$, $\gamma$, $\delta$ and $s\geq 2$
such that
$$
\la Ay+ B(y+w),y\ra \leq -\alpha
\|y\|_{\gamma_1}^{2}+\beta\|w\|_{\gamma_2}^{s}\cdot\|y\|_{\gamma_1}^2+\gamma\|w\|_{\gamma_2}^s+\delta
$$
\mbox{for all $ y\in D(A)$, $w\in V_{\gamma_2}$}.
\end{enumerate}

\noindent
For $\lambda>0$ consider the following decomposition
\begin{equation}
X(t) = Y_{\lambda}(t)+W_{A-\lambda}(t),\quad t\geq 0\, ,
\end{equation}
of the mild solution. It is then easy to see that $Y_\lambda (t)$ satisfies the following semilinear
evolution equation in the mild sense
$$
dY_{\lambda}(t) = \Big(A Y_{\lambda}(t) + \lambda W_{A-\lambda}(t)\Big)dt+B (Y_{\lambda}(t)+W_{A-\lambda}(t))dt
$$
with the random time-dependent nonlinearity $B( \cdot + W_{A-\lambda} (t))$.

\begin{lemma}
\label{lemma1} For any positive increasing $C^1$-function $\Psi$ on
$\R^+$ we have
\begin{equation}
\frac 12 \frac{d}{dt}\Psi (\|Y_{\lambda}(t)\|^2)\leq
- \frac{\alpha}{4}\Psi^{\prime}(\|Y_{\lambda}(t)\|^2)\|Y_{\lambda}(t)\|^2_{\gamma_1}
+ \Psi^{\prime}(\|Y_{\lambda}(t)\|^2)R_{\lambda}(t).
\end{equation}
Where
$$
R_{\lambda}(t) =
\delta+\gamma\|W_{A-\lambda}(t)\|_{\gamma_2}^s+\frac{\lambda^2}{2\alpha}
\|W_{A-\lambda}(t)\|_{-\gamma_1}^2\, ,\quad\lambda = \left( \frac
4\alpha (\beta M_T(\gamma_2,s)+1)\right)^{\frac 1\ve s} \, .
$$
\end{lemma}

\begin{proof}
We have for all $\lambda\geq 0$
\begin{equation*}
\begin{split}
\frac 12 \frac{d}{dt} \Psi(\|Y_{\lambda}(t)\|^2)
& =\Psi^{\prime}(\|Y_{\lambda}(t)\|^2) \la AY_{\lambda}(t)+\lambda W_{A-\lambda}(t)+B(Y_{\lambda}(t)
           +W_{A-\lambda}(t)),Y_{\lambda}(t)\ra \\
& \leq \Psi^{\prime} (\|Y_{\lambda}(t)\|^2)\Big(-\alpha\|Y_{\lambda}(t)\|^2_{\gamma_1}
   +\beta\|W_{A-\lambda}(t)\|_{\gamma_2}^s \cdot\|Y_{\lambda}(t)\|^2_{\gamma_1} \\
& \qquad +\gamma \|W_{A-\lambda}(t)\|_{\gamma_2}^s+\delta+\lambda \la W_{A-\lambda}(t),Y_{\lambda}(t)\ra \Big) \\
& \leq \Psi^{\prime}(\|Y_{\lambda}(t)\|^2) \Big(-\alpha\|Y_{\lambda}(t)\|^2_{\gamma_1}
     + \beta\|W_{A-\lambda}(t)\|_{\gamma_2}^s\cdot\|Y_{\lambda}(t)\|^2_{\gamma_1}  \\
& \qquad + \gamma \|W_{A-\lambda}(t)\|_{\gamma_2}^s + \delta + \frac\alpha 2
     \|Y_{\lambda}(t)\|^2_{\gamma_1} + \frac{\lambda^2}{2\alpha}\|W_{A-\lambda}(t)\|_{-\gamma_1}^2 \Big) \\
& \leq \Psi^{\prime}(\|Y_{\lambda}(t)\|^2) \left(-\frac\alpha 2 \|Y_{\lambda}(t)\|^2_{\gamma_1}
     + \beta \|W_{A-\lambda}(t)\|_{\gamma_2}^s \cdot\|Y_{\lambda}(t)\|^2_{\gamma_1}+R_{\lambda}(t)\right)
\end{split}
\end{equation*}
with
$$
R_{\lambda}(t)=\gamma \|W_{A-\lambda}(t)\|_{\gamma_2}^s+\delta+
\frac{\lambda^2}{2\alpha}\|W_{A-\lambda}(t)\|_{-\gamma_1}^2\, .
$$
Hence by using $\bf{(H_4)}$ and Corollary \ref{Cor} we can write
$$
\|W_{A-\lambda}(t)\|_{\gamma_2}^s\leq \lambda^{-\ve
s}M_T(\gamma_2,s)
$$
with
$$
\E(M_T^m (\gamma_2,s))<\infty\quad\mbox{for all $m\geq 1$}.
$$

\noindent
Thus
\begin{equation*}
\frac 12 \frac{d}{dt}\Psi (\|Y_{\lambda}(t)\|^2)\leq
\Psi^{\prime}(\|Y_{\lambda}(t)\|^2)\left(-\frac\alpha 2
\|Y_{\lambda}(t)\|^2_{\gamma_1}+\beta\lambda^{-\ve
s}M_T(\gamma_2,s)\|Y_{\lambda}(t)\|^2_{\gamma_1}+R_{\lambda}(t)\right).
\end{equation*}

\noindent In particular for $\lambda\defeq \left( \frac 4\alpha
(\beta M_T(\gamma_2,s)+1)\right)^{\frac 1\ve s}$ we have
$$
\frac 12 \frac{d}{dt}\Psi (\|Y_{\lambda}(t)\|^2)\leq
\Psi^{\prime}(\|Y_{\lambda}(t)\|^2)\left(-\frac\alpha 4
\|Y_{\lambda}(t)\|^2_{\gamma_1}+R_{\lambda}(t)\right),
$$
which yields the proof of the lemma.
\end{proof}

\begin{proposition}
\label{lemma3}
Let $0\leq\gamma_{1} \leq\gamma_{2}$ and let $\mu$ be the distribution of any stationary
mild solution of \eqref{sde0}. Then
\[
\int \| x \|^p \mu (dx) < \infty \quad \forall p\ge 0\, .
\]
\end{proposition}

\begin{proof}
First note that for any $q > 0$ there exist positive constants $D_1$, $D_2$ and $D_3$ such that for
%$q=\frac{p}{p-1}$ we can estimate
\begin{equation}
\label{Ineq4}
\begin{aligned}
\E\Big(R_{\lambda}(t)^q\Big)
& \leq D_1+D_2\E\left(\|W_{A-\lambda}(t)\|^{sq}_{\gamma_2}\right)+ D_3\E\left(M_T(\gamma_2,s)^{\frac
2\ve}\|W_{A-\lambda}(t)\|^2_{-\gamma_1}\right)^q \\
& \le D_1 + D_2\E\left(\|W_{A-\lambda}(t)\|^{sq}_{\gamma_2}\right)+ D_3\frac{s-2}{s} \E\left(M_T(\gamma_2,s)^{\frac
2\ve\frac{qs}{s-2}}\right) \\
& \qquad + D_3\frac{2}{s}\E \left( \|W_{A-\lambda}(t)\|_{-\gamma_1}^{sq} \right) \, .
\end{aligned}
\end{equation}
Since $\E\left( \|W_{A-\lambda}(t)\|^{sq}_{-\gamma_1}\right) \leq
\E\left( \|W_A(t)\|^{sq}_{\gamma_2} \right) < \infty$, inequality
\eqref{Ineq4} now implies that $\E\Big(R_{\lambda}(t)^q\Big)$ is
locally integrable w.r.t. $t$.

\medskip
\noindent
For the proof of the moment estimate let us first consider $p \in [0,1]$ and define
$\Psi (t) := \left( 1+ t\right)^{\frac p2}$. Then Lemma \ref{lemma1} implies that
\[
\frac{d}{dt} \left( 1+ \|Y_{\lambda} (t) \|^2\right)^{\frac p2}\leq
- C_{1} \|Y_\lambda (t)\|^2 \left( 1 + \|Y_{\lambda}
(t)\|^2\right)^{\frac p2}  + C_{2} R_{\lambda} (t)
\]
for finite strictly positive constants $C_1$, $C_2$. Fix $K > 0$ and define
$\Psi_{K} (t) := (1+t)^{\frac p2} \wedge K$, $\Phi_{K} (t) := 1_{\{ (1+t)^{\frac{p}{2}} \leq K \}}
t(1+ t)^{\frac p2 -1}$. Then
\[
\frac{d}{dt} \Psi_{K}(\|Y_{\lambda} (t) \|^{2} )
\leq - C_{1} \Phi_{K} ( \|Y_{\lambda} (t) \|^{2}) + C_{2} R_{\lambda} (t)
\]
again, hence
\begin{equation*}
\begin{aligned}
\Psi_{K} (\|Y_{\lambda} (t) \|^{2} ) & + C_{1} \int_{0}^{t} \Phi_{K} ( \|Y_{\lambda} (s) \|^{2})\, ds \\
& \leq \Psi_{K} (\|X (0) \|^{2}) + C_{2} \int_{0}^{t} R_{\lambda} (s)\, ds\, .
\end{aligned}
\end{equation*}

\medskip
\noindent Since for $0\leq p\leq 1$ we have  $(1+ (s+t)^2 )^{\frac
p2} \leq (1+s^2)^{\frac p2} + t^p$ for all $s,t \geq 0$, we conclude
that
\begin{equation*}
\begin{aligned}
\Psi_{K} & (\|X (t) \|^{2} ) + C_{1} \int_{0}^{t} \Phi_{K} ( \|Y_{\lambda} (s) \|^{2})\, ds \\
& \leq \left( (1 + \| Y_{\lambda}(t)\|^2)^{\frac p2} + \| W_{A - \lambda} (t) \|^p \right) \wedge K
      + C_{1} \int_{0}^{t} \Phi_{K} (\| Y_{\lambda} (s) \|^{2})\, ds\\
& \leq \left( 1+ \|Y_{\lambda} (t) \|^2 \right)^{\frac p2} \wedge K
      + C_{1} \int_{0}^{t} \Phi_{K} (\|Y_{\lambda} (s)\|^{2})\, ds
      + \|W_{A- \lambda} (t) \|^{p}\\
& \leq \Psi_{K} (\|X (0)\|^{2}) + C_{2} \int_{0}^{t} R_{\lambda} (s)\, ds + \| W_{A - \lambda} (t) \|^{p}.
\end{aligned}
\end{equation*}
Taking expectations and using stationarity of $(X(t))_{t\ge 0}$ yields the inequality
$$
C_{1} \int_{0}^{t} \E\left( \Phi_{K} (\|Y_{\lambda} (s)\|^{2}) \right) \, ds
\le C_{2} \int_{0}^{t} \E\left( R_{\lambda} (s) \right)\, ds
+ \E\left( \|W_{A- \lambda} (t) \|^{p} \right) < \infty\, .
$$
Since the right hand side does not depend on $K$, we can now take the limit $K\to\infty$ to conclude that
$$
\int_0^t \E\left( \|Y_\lambda (s)\|^2 \left( 1+ \|Y_\lambda (s)\|^2 \right)^{\frac p2 -1}\right)\, ds < \infty
$$
hence
\[
\int_{0}^{t} \E \left( \| Y_{\lambda} (s)\|^{p} \right) \, ds < \infty
\]
too, so that
\begin{equation*}
\begin{aligned}
t \int \|x\|^{p} \mu (dx) & = \int_{0}^{t} \E (\| X (s)\|^{p})\, ds\\
& \leq 2^{p} \int_{0}^{t} \E (\| Y_{\lambda} (s)\|^{p}) \, ds
   + 2^{p} \int_{0}^{t} \E (\|W_{A- \lambda} (s)\|^{p})\, ds\\
& < \infty.
\end{aligned}
\end{equation*}

\medskip
\noindent
For the general case $p > 1$ we proceed by induction. Suppose the assumption is proven for $p$ with
$2p\leq n$ and consider now $p>1$ with $2p\leq n+1$. Lemma \ref{lemma1} now implies that for finite
strictly positive constants $C_1$, $C_2$ and $C_p$
\[
\begin{aligned}
\frac{d}{dt} \left( 1+ \|Y_{\lambda} (t) \|^2 \right)^{\frac p2}
& \leq - C_{1} \|Y_\lambda (t)\|^2 \left( 1 + \|Y_{\lambda} (t)\|^2\right)^{\frac p2 -1}
+ C_{2} \left( 1 + \|Y_\lambda (t)\|^2\right)^{\frac p2 -1} R_{\lambda} (t)  \\
& \le  - C_{1} \|Y_\lambda (t)\|^2 \left( 1 + \|Y_{\lambda} (t)\|^2\right)^{\frac p2 -1}
+ C_p \left( \|Y_\lambda (t)\|^{p-1} + R_{\lambda} (t)^{p-1}  + 1\right) \, .
\end{aligned}
\]
Fix $K > 0$ and let $\Psi_K$ and $\Phi_K$ be as above, the last inequality now implies that
\begin{equation*}
\begin{aligned}
\Psi_{K} (\|Y_{\lambda} (t) \|^{2} ) & + C_{1} \int_{0}^{t} \Phi_{K} ( \|Y_{\lambda} (s) \|^{2})\, ds \\
& \leq \Psi_{K} (\|X (0) \|^{2}) + C_p \int_{0}^{t} \left( \|Y_\lambda (s)\|^{p-1} + R_\lambda (s)^{p-1}
   + 1 \right)\, ds \, .
\end{aligned}
\end{equation*}
Note that for $p>1$ there exists a finite positive constant $C_3$ such that
$$
\left( 1+ (s+t)^2\right)^{\frac p2}
\le \left( 1 + s^2 \right)^{\frac p2} + C_3 (s^{p-\frac 12} + t^{2p-1} +1)
$$
for all $s,t\ge 0$, so that the last inequality now implies that
\begin{equation*}
\begin{aligned}
\Psi_{K} (\|X (t) \|^{2} ) & + C_{1} \int_{0}^{t} \Phi_{K} ( \|Y_{\lambda} (s) \|^{2})\, ds\\
& \leq \left( 1+ \|Y_{\lambda} (t)\|^2 \right)^{\frac p2} \wedge K + C_{1} \int_{0}^{t} \Phi_{K}
  (\|Y_{\lambda} (s)\|^{2})\, ds\\
& \qquad + C_3 \left( \| Y_\lambda (t)\|^{p -\frac 12} + \|W_{A- \lambda} (t)\|^{2p-1 }+ 1\right) \\
& \leq \Psi_{K} ( \|X (0)\|^{2})
  + C_p \int_{0}^{t} \left( \|Y_\lambda (s)\|^{p-1} + R_\lambda (s)^{p-1} + 1 \right)\, ds \\
& \qquad + C_3 \left( \| Y_{\lambda} (t)\|^{p -\frac 12} + \|W_{A- \lambda} (t)\|^{2p-1} + 1\right)\, .
\end{aligned}
\end{equation*}
Taking expectations, using stationarity of $(X(t))_{t\ge 0}$ and the fact that
$$
 \E\left( \|Y_\lambda (t)\|^{p-\frac 12}\right) + \int_0^t \E\left( \|Y_\lambda (s)\|^{p-1} \right) \, ds  < \infty
$$
by assumption on $p$, we conclude that
\begin{equation*}
\begin{aligned}
C_{1} & \int_{0}^{t} \E\left( \Phi_K \left( \| Y_{\lambda} (s)\|^{2} \right)\right) \, ds
\leq C_p \int_{0}^{t} \E \left( \|Y_\lambda (s)\|^{p-1} + R_\lambda (s)^{p-1} + 1 \right)\, ds \\
& \qquad + C_3 \left( \E \left( \| Y_{\lambda} (t)\|^{p -\frac 12} \right)
   + \E \left( \|W_{A- \lambda} (t)\|^{2p-1}\right) + 1\right) < \infty \, .
\end{aligned}
\end{equation*}
Again, the right hand side does not depend on $K$, hence taking the limit $K\to\infty$ we conclude that
\[
\int_{0}^{t} \E (\|Y_\lambda (s)\|^2 (1 + \|Y_{\lambda}
(s)\|^2)^{\frac p2 -1}\, ds < \infty\, ,
\]
hence
\[
\int_{0}^{t} \E (\|Y_{\lambda} (s)\|^p )\, ds < \infty
\]
and thus $\int \|x\|^{p} \mu (dx) < \infty$ too.
\end{proof}

\noindent
Our first main result in this paper now is the following:

\begin{theorem}
\label{mainresult} Let $\gamma_1\leq \gamma_2$ and assume hypotheses
${\bf (H_0)}$-${\bf (H_5)}$ hold. Then the invariant distribution
$\mu$ of any stationary mild solution $(X(t))_{t\ge 0}$ of
\eqref{sde0} satisfies the following moment estimates:
\begin{enumerate}
\item [(i)] $\int \|x\|^{2p}\:\mu(dx)<\infty$ for $p\ge 0$.
\item [(ii)] $\int \|x\|^2_\sigma\|x\|^{2p}\:\mu(dx)<\infty$ for $p\ge 0$, $\sigma<\gamma_1$.
\end{enumerate}
\end{theorem}

\begin{proof}
Clearly, (i) follows from the previous Proposition. For the proof of
(ii) note that Lemma \ref{lemma1} implies that for $\Psi(t)=t^p$
where $p\geq 1$
\begin{equation}
\label{B1}
\begin{split}
\|Y_{\lambda}(t)\|^{2p}+\frac{\alpha}{2}p\int_0^t
\|Y_{\lambda}(s)\|^{2(p-1)}\|Y_{\lambda}(s)\|^2_{\gamma_1}\:ds &\leq
\|x\|^{2p}+2p\int_0^t
\|Y_{\lambda}(s)\|^{2(p-1)}R_{\lambda}(s)\:ds\\
&\hskip-3em\leq  \|x\|^{2p}+2(p-1)\int_0^t
\|Y_{\lambda}(s)\|^{2p}\:ds+2\int_0^t R_{\lambda}(s)^p\:ds.
\end{split}
\end{equation}
From the interpolation inequality
$$
\|x\|_{\sigma}\leq
C\|x\|_{0}^{\frac{\gamma_1-\sigma}{\gamma_1}}\|x\|_{\gamma_1}^{\frac{\sigma}{\gamma_1}}
$$
and Young's inequality, there exist positive constants $C$, $C_1$, $C_2$ such that
\begin{equation}\label{B2}
\begin{split}
\int_0^t
\|W_{A-\lambda}(s)\|^{2(p-1)}\|Y_{\lambda}(s)\|_{\sigma}^2\:ds&\leq
C\int_0^t\|W_{A-\lambda}(s)\|^{2(p-1)}\|Y_{\lambda}(s)\|_{0}^{2\frac{\gamma_1-\sigma}{\gamma_1}}
\|Y_{\lambda}(s)\|_{\gamma_1}^{2\frac{\sigma}{\gamma_1}}\:ds\\
& \hskip -6em\leq C_1\int_0^t
\|Y_{\lambda}(s)\|_{0}^{2\frac{\gamma_1-\gamma}{\sigma}}\|Y_{\lambda}(s)\|_{\gamma_1}^2\:ds+C_2
\int_0^t\|W_{A-\lambda}(s)\|^{2(p-1)\frac{\gamma_1}{\gamma_1-\sigma}}\:ds
\end{split}
\end{equation}
and
\begin{equation}\label{B3}
\int_0^t
\|Y_{\lambda}(s)\|_{0}^{2(p-1)}\|W_{A-\lambda}(s)\|_{\gamma_1}^2\:ds\leq
\frac 12 \int_0^t \|Y_{\lambda}(s)\|_{0}^{4(p-1)}\:ds
+\frac 12\int_0^t \|W_{A-\lambda}(s)\|_{\gamma_1}^4\:ds.
\end{equation}

\noindent
Putting this together with \eqref{B1} and \eqref{B2} yields
\begin{equation*}
\begin{aligned}
\int_0^t \|X(s)\|_{0}^{2(p-1)}\|X(s)\|_{\sigma}^2\:ds
& \leq C_1 \|X(0)\|^{2p}_0 + C_2 \int_0^t\|X(s)\|_0^{2p_1}\:ds + C_3 \int_0^t R_{\lambda}(s)^{p_2}\:ds
 \\
& \qquad + C_4 \int_0^t \|W_{A-\lambda}(s)\|_{\gamma_2}^{2p_3}\:ds,
\end{aligned}
\end{equation*}
for some constants $p_i$ and $C_i$.

\noindent
Taking expectations we obtain that
$$
t\int_H \|x\|^{2(p-1)}\|x\|_{\sigma}^2\:\mu(dx)=\E\left(
\int_0^t\|X(s)\|^{2(p-1)}\|X(s)\|_{\sigma}^2\:ds\right)<\infty \, .
$$
hence the assertion.
\end{proof}

%%%%%%%%%%%%%%%%%%%% Section 4 %%%%%%%%%%%%%%%%%%%%%%%%%%%%%%%%%%%%%

\section{Maximal dissipativity of the Kolmogorov operator}

\medskip
\noindent In the previous section we discussed a priori estimates of
invariant measures $\mu$ for the equation \eqref{sde0}. Suppose for
the moment that \eqref{sde0} has a unique mild solution $X(t,x)$,
$t\ge 0$, for any initial condition $x\in H$, that $x\mapsto X(t,x)$
is measurable for any $t$ and that the stationary solution $X(t)$,
$t\ge 0$, of \eqref{sde0} can be represented as $X(t) = X(t, X_0)$,
$t\ge 0$. Furthermore we take $Q=(-A)^{2\gamma_0}$ for some
$\gamma_0< \frac{1}{2}$. It is then easy to see that in this case,
the associated transition semigroup
$$
P_t \varphi (x) = E\left( \varphi (X(t,x))\right) \, , \varphi\in\cB_b (H)\, ,
$$
induces a $C_0$-semigroup of Markovian contractions $(\tilde{P}_t )_{t\ge 0}$
on $L^1 (H, \mu )$, in fact on any $L^p (H, \mu )$ for $p\in [1, \infty [$. In the case where
$$
\langle B(x), h\rangle \, , \langle x, h\rangle \in L^1 (\mu ) \quad \mbox{ for any }\quad h\in D(A) \, ,
$$
the corresponding infinitesimal generator $L$ has the expression
$$
L\varphi (x) = \frac 12 \Tr_{H} \left( \sqrt{Q}D^2\varphi
(x)\sqrt{Q}\right) + \la x, A D\varphi (x)\ra + \la B(x), D\varphi
(x)\ra\, \qquad \varphi\in\test\, .
$$
Here,
$$
\begin{aligned}
\test & := \Big\{ \varphi\in C_b^2 (H) \mid \varphi (x) = f (\langle x,h_1\rangle ,\ldots , \langle x,h_m\rangle ),
f\in C_b^2 (\R^m) , m\ge 1,\\
& \qquad\qquad\qquad\qquad h_1, \ldots , h_m\in D(A)\Big\}
\end{aligned}
$$
denotes the space of suitable cylindrical test functions (see Proposition 3.1 in \cite{Es-St} for a proof).
As an application of the improved moment estimates on $\mu$, obtained in the last section, we shall discuss
in this section whether $(\tilde{P}_t)_{t\ge 0}$ is the only $C_0$-semigroup in $L^1 (H, \mu )$ whose infinitesimal
generator extends $(L, \test )$. In this case we say that $L$ is $L^1$-unique.

\medskip
\noindent
In the general case, the mere existence of a stationary solution of \eqref{sde0} neither
ensures the existence of the associated transition semigroup $(P_t)_{t\ge 0}$ nor the
existence of its $L^1$-counterpart $(\tilde{P}_t)_{t\ge 0}$, but only implies that the
measure $\mu$ is infinitesimally invariant for $L$, i.e.,
$$
\int_H L\varphi(x)\:\mu(dx)=0
$$
for all $\varphi\in\test$ with $L\varphi\in L^1 (H, \mu )$.

\medskip
\noindent
However, in this case, $(L,\test )$ is dissipative, in particular closable, in $L^1 (H, \mu )$ (see \cite{Es-St}).
Therefore, to obtain the existence (and also the uniqueness) of $(\tilde{P}_t)_{t\geq0}$, it is sufficient to
prove that the closure of $L$ in $L^1 (H, \mu )$ generates a $C_0$-semigroup. The $L^1$-counterpart
$(\tilde{P}_t )_{t\ge 0}$ will be Markovian and its existence can therefore be regarded as a first necessary step
in the construction of a full Markov process associated with \eqref{sde0}.

\medskip
\noindent
For our analysis in this section we need the following assumptions:
\begin{enumerate}
\item[$\bf{(A_0)}$] The measure $\mu$ is infinitesimally invariant for
$L$.
\item[$\bf{(A_1)}$] $\|B\|\in L^1 (H,\mu )$, where the vector field $B: D(B)\subset H\to H$ is considered
as a vector field on all of $H$ by setting $B(x) = 0$ if $x\in H\setminus D(B)$.
\item[$\bf{(A_2)}$] For some $\beta \in (\gamma_0,\frac 12)$, there exists $C :\:V_{\beta}\rightarrow V_{\beta}$ with $\int\|C(x)\|^{2}_{\beta}
\:\mu(dx)<+\infty$ such that
\begin{equation}
\la B(x)-B(y),x-y\ra\leq\|x-y\|_{\frac 12}^2 + \la C(x) -
C(y),x-y\ra \quad \forall x, y \in V_{\frac 12}
\end{equation}
\end{enumerate}

\medskip
\noindent
In the following, let us define finite dimensional Galerkin approximations for $L$. To this end let
$$
i_n:\:\R^n\longrightarrow H,\quad
(x_1,\ldots,x_n)\mapsto\sum\limits_{k=1}^{n}x_ke_k
$$
be the natural injection of $\R^n$ into $H$ and
$$
\pi_n:\:H\longrightarrow\R^n,\quad x\mapsto (\la x,e_1\ra,\ldots,\la
x,e_n\ra)
$$
the natural projection of $H$ on $\R^n$. Let
$$
A^n\defeq \pi_n\circ A\circ i_n:\:\R^n\rightarrow \R^n
$$
and
$$
B^n\defeq \pi_n\circ B\circ i_n:\:\R^n\rightarrow \R^n,\quad
C^n\defeq \pi_n\circ C\circ i_n:\:\R^n\rightarrow \R^n
$$
be the corresponding operator and vector-fields induced by $A$, $B$
and $C$ on $\R^n$ and consider the Kolmogorov operator
$$
L^n\varphi(x)\defeq \frac 12 \sum\limits_{k=1}^n  \la
(-A^{n})^{-2\gamma_{0}} e_{k}, e_{k} \ra \varphi
_{x_kx_k}(x)+\sum\limits_{k=1}^n\la A^n
x+B^n(x)-C^n(x),e_k\ra\varphi_{x_k}(x),\quad \varphi\in C_b^2(\R^n).
$$

\noindent
We now make the following additional assumption on $L^n$.
\vskip 1em
\begin{itemize}
\item[${\bf (A_3)}$] For $n\geq 1$, $B^n$ and $C^n$ are smooth,
polynomially bounded vector-fields.
\end{itemize}
\vskip 1em

\noindent
Note that ${\bf (A_2)}$ now implies the one-sided Lipschitz condition
\begin{equation}\label{diss-B_n}
\la \left(A^n x + B^n (x) - C^n (x)\right) -\left( A^n y + B^n(y) - C^n (y)\right),x-y\ra \leq 0 \quad x,\:y\in \R^n,
\end{equation}
for the finite-dimensional approximations of $Ax + B(x) - C(x)$.

\medskip
\noindent Next, let $U:\: H\rightarrow V_{\beta}$ be a smooth vector
field that is Lipschitz continuous w.r.t the H-norm with Lipschitz
constant $\Lip_{U}$ and denote by $L^n_U$ the Kolmogorov operator
$$
L^n_U\varphi(x) = L^n\varphi(x)+\sum\limits_{k=1}^n \la
U^n(x),e_k\ra \varphi_{x_k}(x),\quad\varphi\in C_b^2(\R^n),
$$
where $U^n=\pi_n\circ U\circ i_n:\:\R^n\rightarrow\R^n,\quad n\geq 1$. \eqref{diss-B_n} now implies the
one-sided Lipschitz condition
$$
\la \left(A^n x + B^n(x) - C^n (x)+U ^n(x)\right)-\left(A^n y + B^n(y) - C^n (x)+U^n(y)\right),x-y\ra\leq
\Lip_{U} \|x-y\|^2 \, ,
$$
$x,\:y\in \R^n$, which is equivalent with
\begin{equation}
\label{diss}
\la \left( A^n + D\left( B^n -C^n + U^n \right)\right) \xi , \xi\ra \leq \Lip_{U}\|\xi\|^2 \qquad \forall\,\xi\in\R^n\, .
\end{equation}

\medskip
\noindent
Since the coefficients of $L^n_U$ are smooth there exists for any $f\in C_b^2(\R^n)$ a solution
$C([0,+\infty)\times\R^n)\cup C^{1,2}_{loc}((0,+\infty)\times \R^n)$
of the Cauchy-problem
\begin{equation}\label{CP}
 \left\{
\begin{array}{ll}
du(t,x)= L^n_U u(t,x)dt,\quad\mbox{for}\:\: (t,x)\in (0,+\infty)\times\R^n\\
u(0,x)=f(x),\:x\in\R^n,
\end{array}
\right.
\end{equation}
satisfying $\|u\|_{\infty}\leq \|f\|_{\infty}$ (and $u\geq 0$ if $f\geq 0$). In addition, there exists
a semigroup of linear operators $(T_t^{U^n})_{t\geq 0}$ on $C_b(\R^n)$ such that for $f\in C_b(\R^n)$
the solution of \eqref{CP} is represented as
$$
u(t,x)=T_t^{U^n}f(x),\quad t\geq 0,\:x\in\R^n
$$
(see Theorem 2.2.5 in \cite{BL}). According to Theorem 6.1.7 in \cite{BL} we also have the norm-estimates
\begin{equation}
\label{uniform-estimate}
\|T_t^{U^n}f\|_{C_b^1(\R^n)}\leq C \|f\|_{C_b^1(\R^n)},\quad
f\in C_b^1(\R^n)
\end{equation}
for some uniform constant $C>0$. A simple coupling argument shows that the constant in \eqref{uniform-estimate}
may be chosen to be $e^{\Lip_{U}t}$, taking into account \eqref{diss}. Note that this constant is independent of
$n$, $n\geq 1$.

\medskip
\noindent
In the following we will use the notation $``\bar{\varphi}``$ for $\varphi\in\mathcal{B}(\R^n)$
to denote the function $\bar{\varphi}=\varphi\circ \pi_n$. Then
$$
\la D\bar{\varphi}(x), e_k \ra =
\begin{cases}
\varphi_{x_k}(x) & \mbox{if}\:\:k=1,\ldots, n\\
0 &  \mbox{otherwise}
\end{cases}
$$
and $\bar{\varphi}\in \test$ if $\varphi\in C_b^2(\R^n)$. In particular
$\overline{T_t^{U^n}f}\in \test$ for $t\geq 0$, $f\in C_b^2(\R^n)$. We will also use
the notation
$$
\|x\|_\alpha := \|i_n x\|_\alpha \, , x\in\R^n\, , \alpha\in\R\, .
$$

\medskip
\noindent
The following a priori estimate is crucial.

\begin{lemma}
\label{lemma4_1}
Let $f\in C_b^2(\R^{n_0}) \mbox{ and }\lambda > 0$. Then for $n\geq n_0$ we have
\begin{equation}\label{l2-gradientestimate}
\begin{split}
\int_0^t e^{-\lambda s}\int_H & \|D
\overline{T^{U^n}_s f}\|_{-\gamma_0}^2\:d\mu\:ds
\leq 4 \frac{e^{\Lip_{U}t}}{\lambda} \int_H\|B-B^n\|\:d\mu\cdot\|f\|^2_{C_b^1(\R^n)} \\
& \qquad + 2\|f\|^2_{\infty}+\frac {4}{\lambda}
\int_H\|C^n-U^n\|^2_{\gamma_0}\:d\mu\|f\|_{\infty}^2.
\end{split}
\end{equation}
\end{lemma}

\begin{proof}
Clearly, invariance of $\mu$ implies for $\varphi\in\test$ that
\begin{equation*}
\begin{split}
\frac 12\int_H & \|D \varphi \|_{-\gamma_0}^2\:d\mu
   = -\int_H L\varphi\, \varphi\, d\mu \\
& = -\int_H L^n_U \varphi\, \varphi\, d\mu
   -\int_H \la B-B^n+C^n-U^n, D \varphi\ra\varphi\:d\mu \\
& \leq -\int_H L^n_U\varphi\, \varphi\,d\mu
  + \| D \varphi\|_{\infty} \|\varphi\|_{\infty}\int_H\|B-B^n\|\:d\mu \\
& \quad +\left(\int_H\|C^n-U^n\|^2_{\gamma_0}\:d\mu\right)^{\frac 12}\,
\left(\int_H\| D \varphi\|^2_{-\gamma_0}\:d\mu \right)^{\frac 12}\cdot \|\varphi\|_{\infty}
\end{split}
\end{equation*}
and thus
\begin{equation}\label{gradient-est}
\begin{split}
\int_H \| D \varphi\|_{-\gamma_0}^2\:d\mu
& \leq -4\int_H L^n_U\varphi \, \varphi\, d\mu+4\| D \varphi\|_{\infty}
   \|\varphi\|_{\infty}\int_H\|B-B^n\|\:d\mu \\
& \qquad + 4\|\varphi\|^2_{\infty}\int_H\|C^n-U^n\|^2_{\gamma_0}\:d\mu\, .
\end{split}
\end{equation}

\noindent
Inserting $\overline{T_s^{U^n}f}$ in \eqref{gradient-est}, using
$\| D \overline{T_s^{U^n}f}\|_{\infty} = \| D T^{U^n}_sf\|_{\infty}
\leq e^{\Lip_{U}s}\|f\|_{C_b^1(\R^{n_0})}$ and $L^n_U T_s^{U^n}f
=\frac{d}{ds}T_s^{U^n}f$, we obtain that
\begin{equation}
\label{gradient-est-semigroup}
\begin{aligned}
\int_H \| D\overline{T_s^{U^n}f}\|_{-\gamma_0}^2\:d\mu
& \leq 4\int_H\|B-B^n\|\:d\mu\cdot e^{\Lip_{U}s} \|f\|^2_{C_b^1(\R^{n_0})}\\
& - 2\int_H  \frac{d}{ds}(\overline{T_s^{U^n}f})^2\: d\mu
  + 4\|f\|^2_{\infty}\int_H\|C^n-U^n\|^2_{\gamma_0}\:d\mu \, .
\end{aligned}
\end{equation}
Multiplying both sides of the above inequality by $e^{-\lambda s}$ and using
$$
\frac{d}{ds}\left(e^{-\lambda s}(\overline{T_s^{U^n}f})^2\right)
\leq e^{-\lambda s}\frac{d}{ds}(\overline{T_s^{U^n}f})^2
$$
we conclude for $s\le t$ that
\begin{equation*}
\begin{aligned}
e^{-\lambda s}\int_H \|D\overline{T_s^{U^n}f}\|_{-\gamma_0}^2\:d\mu
&  \leq - 2\int_H  \frac{d}{ds}\left(e^{-\lambda s}(\overline{T_s^{U^n}f})^2\right)\: d\mu \\
& \quad + 4 e^{-\lambda s} \int_H\|B-B^n\|\:d\mu\cdot e^{\Lip_{U}t} \|f\|^2_{C_b^1(\R^{n_0})}\\
& \quad + 4 e^{-\lambda s} \|f\|^2_{\infty}\int_H\|C^n-U^n\|^2_{\gamma_0}\:d\mu\, .
\end{aligned}
\end{equation*}
Integrating the last inequality with respect to $ds$ yields inequality \eqref{l2-gradientestimate}.
\end{proof}

\begin{lemma}
Let $\lambda >0$ and  $h\in\mathcal{B}_b (H)$ be such that
$$
\int_H (\lambda-L)\varphi \, h\:d\mu = 0 \qquad \mbox{for all } \varphi\in \test\, ,
\varphi=f\circ\pi_{n_0}\, , f\in C_b^2(\R^{n_0})\, .
$$
Then for $n\ge n_0$
\begin{equation}
\label{EstimateLemma4_2}
\begin{aligned}
 \left| \int_H \varphi \, h\,  d\mu  \right|
&  \le e^{-\lambda t} \|\varphi\|_\infty \|h\|_\infty
 + \|f\|_{C_b^1 (\Bbb R^{n_0})} \|h\|_\infty \frac{e^{\Lip_U t}}{\lambda}
   \int_H \|B - B^n \|\, d\mu  \\
& \qquad + \frac{\|h\|_\infty}{\sqrt{\lambda}} \left( \int_0^t e^{-\lambda s} \int_H
\|D\overline{T_s^{U^n} f} \|_{-\gamma_0}^2\, d\mu \, ds \right)^{\frac 12}
\left( \int_H \| C^n - U^n\|^2_{\gamma_0}\, d\mu  \right)^{\frac 12}\, .
\end{aligned}
\end{equation}
\end{lemma}

\begin{proof}
Since
$$
\frac{d}{ds}T_s^{U^n}f=L^n_U T_s^{U^n}f,\quad
s>0,
$$
it follows for $s\le t$ that
\begin{equation*}
\begin{aligned}
\frac{d}{ds}\:e^{-\lambda s} & \int_H \overline{T_s^{U^n}f}\, h\:d\mu
= e^{-\lambda s}\int_H (L^n_U -\lambda)\overline{T_s^{U^n}f}\, h\:d\mu\\
&\hskip-2em=e^{-\lambda s}\int_H\la
B^n-B-C^n+U^n, D \overline{T_s^{U^n}f}\ra h\:d\mu \\
& \hskip-2em\leq e^{-\lambda s} e^{\Lip_{U}t}\|f\|_{C_b^1(\R^{n_0})}\cdot\|h\|_{\infty}\int_H \|B^n -B\|\:d\mu \\
&  + e^{-\lambda s} \|h\|_\infty \left( \int_H \|C^n - U^n
\|^2_{\gamma_0} \, d\mu \right)^{\frac 12} \left(
\int_H\|D\overline{T_s^{U^n}f} \|_{-\gamma_0}^2\, d\mu
\right)^{\frac 12} \, .
\end{aligned}
\end{equation*}
Integrating the last inequality with respect to $s$ and applying H\"older's inequality to the
second term yields the assertion \eqref{EstimateLemma4_2}.
\end{proof}

\medskip
\noindent
We are now ready to prove the main result

\begin{proposition}
\label{Prop4_3} Let $\lambda>0$ and suppose $h\in\mathcal{B}_b(H)$
is such that
$$
\int_H (\lambda-L)\varphi\, h\:d\mu = 0\quad\mbox{ for all }\quad\varphi\in\test\, .
$$
Then $h=0\quad\mu$-a.e.
\end{proposition}

\begin{proof}
Suppose on the contrary that $h\neq 0$. Then there exists $\varphi=f\circ\pi_n$ for some $f\in C_b^2(\R^n)$ with
$$
\ve\defeq \left|\int_H \varphi \, h\:d\mu \right|>0.
$$
We may suppose that $\ve\le 1$.

\medskip
\noindent Let $U: H\rightarrow V_{\beta}$ be such that $U$ is
Lipschitz w.r.t. the $H$-norm and
$$
\left(\int_H \|C-U \|_{\beta}^2 \:d\mu \right)^{\frac 12}
\leq\frac{\ve}{8\left( 1 + \frac{\|h\|_\infty}{\sqrt{\lambda}}
\left( 2 \|f\|^2_\infty + 1\right)^\frac 12 + \frac 4\lambda
\|f\|^2_\infty \right)}\, .
$$
Since
$$
\lim\limits_{n\to+\infty}\int_H \|U-U^n\|_{\beta}^2\:d\mu +\int_H
\|C-C^n\|_{\beta}^2\:d\mu =0\, ,
$$
and using the fact that $\gamma_0<\beta$ we can find
$n_\varepsilon\ge n_0$ such that
$$
\begin{aligned}
\sup_{n\ge n_\varepsilon} &
\left( \int_H \|U-U^n\|_{\gamma_0}^2\:d\mu \right)^{\frac 12}
+ \left(\int_H \|C-C^n\|_{\gamma_0}^2\:d\mu\right)^{\frac 12} \\
& \le \frac{\ve}{4 \left( \frac{\|h\|_\infty}{\sqrt{\lambda}} \left( 2
\|f\|^2_\infty + 1\right)^{\frac 12} + \frac 4\lambda \|f\|^2_\infty\right) }\, .
\end{aligned}
$$
In particular,
\begin{equation}
\label{EstimateProp4_3_2} \sup_{n\ge n_\varepsilon} \left( \int_H
\|C^n -U^n\|_{\gamma_0}^2\: d\mu \right)^{\frac 12} \le
\frac{\ve}{2\left( \frac{\|h\|_\infty}{\sqrt{\lambda}} \left( 2
\|f\|^2_\infty + 1\right)^\frac 12 + \frac 4\lambda \|f\|^2_\infty \right)}\, .
\end{equation}

\medskip
\noindent
Let $t_\ve > 0$ be such that $e^{-\lambda t_\ve} \|\varphi\|_{\infty} \|h\|_\infty < \frac {\ve}4$. Since
$\lim_{n\to\infty} \int_H \|B - B^n \| \, d\mu = 0$, we can find by Lemma \ref{lemma4_1} and \eqref{EstimateProp4_3_2}
$\tilde{n}_\ve\ge n_\ve$ such that
\begin{equation}
\label{EstimateProp4_3_1}
\sup_{n\ge \tilde{n}_\ve} \int_0^{t_\ve} e^{-\lambda s} \int_H \left\| D\overline{T_s^{U^n}f}(x)
\right\|^2_{-\gamma_0} \, \mu(dx)\, ds \le 2 \|f\|^2_\infty + 1\, .
\end{equation}
Inserting \eqref{EstimateProp4_3_1} into \eqref{EstimateLemma4_2} we obtain for $n\ge \tilde{n}_\ve$ the estimate
$$
\begin{aligned}
\left| \int_H \varphi \,  h\, d\mu  \right| & \le e^{-\lambda t_\ve} \|\varphi\|_\infty \|h\|_\infty \\
& \quad + \|f\|_{C_b^1 (\Bbb R^{n_0})} \|h\|_\infty \frac{e^{\Lip_U t_\ve}}{\lambda}
\int_H \|B - B^n  \| \, d\mu  \\
& \quad + \frac{\|h\|_\infty}{\sqrt{\lambda}} \left(2 \|f\|_\infty^2 + 1\right)^{\frac 12}
\left( \int_H \|U^n - C^n \|^2_{\gamma_0} \, d\mu\right)^{\frac 12} \\
& \le \frac{3\ve}4 + \|f\|_{C_b^1 (\Bbb R^{n_0})} \|h\|_\infty  \frac{e^{\Lip_U t_\ve}}{\lambda}
\int_H \|B - B^n \| \, d\mu
\end{aligned}
$$
where the last inequality follows from \eqref{EstimateProp4_3_2}. Consequently,
$$
\begin{aligned}
\left| \int_H \varphi  \, h \, d\mu  \right|
& \le \limsup_{n\to\infty} \frac{3\ve}4 + \|f\|_{C_b^1 (\Bbb R^{n_0})} \|h\|_\infty
\frac{e^{\Lip_U t_\ve}}{\lambda} \int_H \|B - B^n  \| \, d\mu  =
\frac{3\ve}4 \, ,
\end{aligned}
$$
which is a contradiction to our assumption. Thus $h=0$ $\mu$-a.e. and the proof is complete.
\end{proof}

\medskip
\noindent
We have thus proven the following

\begin{theorem}
\label{uniqueness}
Let $(\bar{L},D(\bar{L}))$ be the closure of $(L,\test)$ in
$L^1(H,\mu)$. Then $(\bar{L},D(\bar{L}))$ generates a
$C_0$-semigroup of contractions $(\bar{P}_t)_{t\geq0}$ on
$L^1(H,\mu)$, $(\bar{P}_t)_{t\geq0}$ is Markovian and the measure
$\mu$ is $(\bar{P}_t)_{t\geq0}$-invariant.
\end{theorem}

\begin{proof}
Proposition \ref{Prop4_3} implies that for $\lambda > 0$ the range $(\lambda - L)(\test )$ is dense in
$L^1 (H, \mu )$, so that $(\lambda - \bar{L})(D(\bar{L})) = L^1 (H, \mu )$. An application of Lumer-Phillips's
theorem (see \cite[Theorem 3.15]{EN}) implies that $\bar{L}$ generates a $C_0$-semigroup $(\bar{P}_t)_{t\ge 0}$
of contractions. The proof of the Markovianity of $(\bar{P}_t)_{t\ge 0}$ and its $\mu$-invariance is exactly the same
as the proof of the corresponding statements in \cite[Theorem 3.4]{Es-St}.
\end{proof}

\begin{cor}
Suppose that \eqref{sde0} has a unique mild solution $X(t,x)$, $t\ge 0$, for any initial condition $x\in H$, and that
$x\mapsto X(t,x)$ is measurable, $t\ge 0$. If the measure $\mu$ is subinvariant for the associated transition semigroup
$$
P_t\varphi (x) = E\left( \varphi (X(t,x))\right) \, , \varphi\in\cB_b (H)\, ,
$$
i.e., $\int P_t \varphi\, d\mu \le \int\varphi\, d\mu$ for all $\varphi\in \cB_b (H)$, $\varphi\ge 0$, then $P_t\varphi$
is a $\mu$-version of $\bar{P}_t\varphi$ for all $\varphi\in\cB_b (H)$.
\end{cor}

\medskip
\noindent
For the proof of the Corollary, it is sufficient to note that under the assumptions made, the transition semigroup
$(P_t)_{t\ge 0}$ induces a $C_0$-semigroup $(\tilde{P}_t)_{t\ge 0}$ on $L^1 (H,\mu )$ whose infinitesimal generator
extends $(L, \test )$. Since the latter is $L^1$-unique, we conclude that $\tilde{P}_t = \bar{P}_t$, $t\ge 0$.

%%%%%%%%%%%%%%%%%%%%%%%%%%%%%% Section 5 %%%%%%%%%%%%%%%%%%%%%%%%%%%%%%%%%%%%%%%%%%%%%%%%%%%%%%%%%%%%%

\section{Application}
\label{SectionExamples}

\subsection{Stochastic Burgers equation}
Let $I = [0, 1]\subset\R$ and $A = \frac{d^2}{dx^2}$ be the
Laplacian with Dirichlet boundary conditions and consider the
stochastic partial differential equation
\begin{equation}
\label{SB} dX(t,x) = \left( \frac{d^2X}{dx^2}(t,x) +
\partial_x(X^2(t,x))\right)\, dt + \eta (t,x) \, , \quad (t,x) \in \R_+\times
I\, ,
\end{equation}
where $\eta (t,x) = dW_t (x)$ and $(W_t)$ is a cylindrical Wiener
process on $L^2 (I)$ with covariance operator $Q = \left( -
A\right)^{-2\gamma_0}$ for some $\gamma_0\in (0,\frac 14)$ fixed.
This implies in particular that the stochastic convolution
corresponding to \eqref{SB} has a continuous version in
$V_{\gamma_0}\defeq D((-A)^{\gamma_0})$. Therefore by using a
similar argument as in \cite[Chap 14]{DaZ2} (see also \cite{HM}) one
can prove the existence of a unique mild solution $X(t)$, $t\geq 0$
of \eqref{SB}. Existence of an invariant probability measure $\mu$
for \eqref{SB} has been shown in \cite{Da-Ga}, \cite{DaZ2}. We shall
mention that in the sequel we will consider $X(t)$, $t\geq 0$ as a
stationary solution for \eqref{SB} (see section 1).

\medskip
\noindent
It is clear that the nonlinear part of the drift term
$B(u)\defeq \partial_x(u^2)$ is neither Lipschitz nor one-sided
Lipschitz. However, it is straightforward to check that
\begin{equation}\label{D}
\la B(u)-B(v),u-v\ra\leq \frac 12 \|u-v\|_{\frac 12}^2+\la
u^3-v^3,u-v\ra,\quad u,\:v\in V_{\frac 12},
\end{equation}

\noindent using the elementary inequality
$$
\frac 12 (a+b)^2(a-b)^2\leq (a^3-b^3)(a-b),\quad a,\:b\in\R.
$$

\noindent
We remark that the coefficients of \eqref{SB} satisfy the following Lyapunov-condition
\begin{equation}
\label{LB} \la Ay+\partial_x(y+w)^2,y\ra \leq -\frac 12 \|y\|_{\frac
12}^2+\alpha\|w\|_{\frac 18}^2\|y\|_{\frac 12}^2+\beta\|w\|_{\frac
18}^4\, , y \in D (A), w\in V_{\frac{1}{2}}\, .
\end{equation}

\noindent
Indeed, for $y$, $w\in V_{\frac 12}$ it follows that
\begin{equation*}
\begin{split}
\int_I\partial_x(y+w)^2\:y\:dx&=-2\int_I \partial_xy\cdot
yw\:dx-\int_I \partial_xy \cdot w^2\:dx\\
&\leq \frac 14\|y\|_{\frac 12}^2+\int_I y^2w^2\:dx+\frac
14\|y\|_{\frac 12}^2+\|w\|_{L^4(I)}^4\\
&\leq \frac 12 \|y\|_{\frac 12}^2+C\|w\|^2_{L^4(I)}\|y\|^2_{\frac
12}+\|w\|^4_{L^4(I)}.
\end{split}
\end{equation*}

\noindent
Using the Sobolev embedding $W^{\frac 14,2}(I)\hookrightarrow L^4(I)$ and the fact that
the norm of  $W^{\frac 14,2}(I)$ and $V_{\frac 18}$ are equivalent, we conclude that
\begin{equation}
\label{Est1}
\int_I \partial_x(y+w)^2\:y\:dx\leq \frac 12 \|y\|_{\frac 12}^2
+\alpha \|w\|^2_{\frac 18}\|y\|^2_{\frac 12}+ \beta \|w\|^4_{\frac 18}
\end{equation}
for suitable constants $\alpha$, $\beta$, hence \eqref{LB} follows.
The eigenvalues of  $-A$ are given by $\lambda_k = + \pi^2 k^2$, $k\geq 1$. It follows from
Corollary \ref{Cor} that for $T>0$ we have
\begin{equation*}
\sup\limits_{0\leq t\leq T}\|W_{A-\lambda}(t)\|_{{\gamma}}^2\leq
\lambda^{-2\ve}\cdot M_{\delta,\gamma,\ve},\quad \lambda>0,
\end{equation*}
for some random variable $M_{\delta,\gamma,\ve}$ with finite moments of any order, if
\begin{equation}
\label{c_1}
\kappa\defeq (\delta+\gamma_0)-(\gamma+\ve)>\frac 14,\quad \delta\in ]0,\frac 12[,
\end{equation}
because then
\begin{equation*}
Z_{\gamma,\delta,\ve} =\sum\limits_{k=1}^{\infty}\lambda_{k}^{-2(\delta-\gamma-\ve)}q_k
=\sum\limits_{k=1}^{\infty}\lambda_k^{-2(\delta+\gamma_0-\gamma-\ve)}
=\pi^{-4\kappa }\sum\limits_{k=1}^{\infty} k^{-4 \kappa}<\infty\, .
\end{equation*}

\noindent
Theorem \ref{mainresult} now implies the following moment estimates
\begin{equation}
\label{M1}
\begin{split}
\int \|u\|_0^{2p}\:\mu(du)< + \infty\,  , & \quad p\ge 0  \\
\int \|u\|_{\sigma}^2\|u\|_0^{2p}\:\mu(du) < +\infty\, ,
& \quad p\ge 0,\:\:\sigma\leq \frac 12\, .
\end{split}
\end{equation}
\noindent
The following Proposition will be crucial for the uniqueness of the Kolmogorov operator
associated with \eqref{SB}.
\begin{proposition}
Let $\beta\in (\frac 14,\gamma_0+\frac 14)$. Then
\begin{itemize}
\item[(i)] $\|B(u)\|\in L^1(\mu)$.
\item [(ii)] $\|u^3\|_{\beta}\in L^2(\mu)$.
\end{itemize}
\end{proposition}

\noindent
The proof is accomplished in the following three Lemmata.
\begin{lemma}
\label{lemma4.1} We have  $\int\|B(u)\|\:\mu(du)<+\infty$.
\end{lemma}
\begin{proof}
First note that $\|u\|_{\infty}\leq C_{\frac 14 + \ve}\|u\|_{\frac 14+\ve}$ for
any $\ve>0$, so that
$$
\|B(u)\|\leq \|u\|_{\frac 12}\|u\|_{\infty}\leq C_{\frac 12}\|u\|^2_{\frac 12}
$$
and now  the moment estimate \eqref{M1} implies the assertion.
\end{proof}
\begin{lemma}
\label{lemma4.2} Let $\beta\in (\frac 14,\frac 14+\gamma_0)$. Then
for $p=1,\:2,\:3$ we have
$$\int\|u^3\|^2_{\beta}\:\mu(du)+\int\|u^p\|^2\:\mu(du)<+\infty.$$
\end{lemma}

\begin{proof}
First note that Sobolev's imbedding, followed by real interpolation,
implies that for $\theta>\frac 38$
\begin{equation*}
\|u\|_{L^p(I)} \leq C\|u\|_{W^{\frac{p-2}{2p},2}(I)}\leq
C\|u\|_{\frac{p-2}{4p}} \leq C \|u\|_{\frac 14+\theta}^{\frac
{p-2}{p(1+4\theta)}}\cdot
\|u\|_{0}^{\frac{2+4p\theta}{p(1+4\theta)}},
\end{equation*}
so that
$$
\|u\|_{L^p(I)}^p\leq C \|u\|_{\frac 14+\theta}^{\frac
{p-2}{1+4\theta}} \cdot \|u\|_{0}^{\frac{2+4p\theta}{1+4\theta}},
$$
and now \eqref{M1} implies that
$$
\int \|u\|_{L^p(I)}^p\:\mu(du)<+\infty\quad \mbox{if}\quad
\frac{p-2}{1+4\theta}\leq 2\Longleftrightarrow p\leq 4(1+2\theta).
$$

\noindent Since $\theta>\frac 38$ implies $4(1+2\theta)>7$, we thus
obtain that
$$
\int \|u^3\|^2\:\mu(du)<+\infty.
$$

\noindent Let us now prove $\int
\|u^3\|^2_{\beta}\:\mu(du)<+\infty$. To this end we consider again
the decomposition
$$
X_t=Y_t+W_A(t),\quad t\geq 0
$$
of the mild solution of \eqref{SB}. Then for $p\geq 1$

\begin{equation}\label{est2}
\begin{split}
\frac{1}{2p}\left(\frac{d}{dt}\|Y_t\|_{L^{2p}(I)}^{2p}\right)&=\int\frac{d^2Y_t}{d^2x}Y_t^{2p-1}\:dx+\int
\frac{d}{dx}\Big(Y_t+W_A(t)\Big)^2Y_t^{2p-1}\:dx\\
& = - (2p-1)\int\Big(\frac{dY_t}{dx}\Big)^2Y_t^{2(p-1)}\:dx-2(2p-1)\int W_A(t)Y_t^{2p-1}\frac{dY_t}{dx}\:dx\\
&\qquad -(2p-1)\int W_A(t)^2\frac{dY_t}{dx}Y_t^{2(p-1)}\:dx \\
&\leq -\frac{(2p-1)}{2}\int \Big(
\frac{dY_t}{dx}\Big)^2Y_t^{2(p-1)}\:dx+C\left(\|W_A(t)\|_{L^8(I)}^8+\|Y_t\|_{L^{2p}(I)}^{2p}\right).
\end{split}
\end{equation}

\noindent
Integrating \eqref{est2} with respect to $t$ we conclude that
\begin{equation}
\label{est3}
\int_0^T \Big \|\Big(\frac{d Y_t}{dx}\Big)\cdot
Y_t^{p-1}\Big\|^2\:dt\leq C_1\|Y_0\|_{L^{2p}(I)}^{2p}+C_2
\int_0^T\|W_A(t)\|_{L^8(I)}^8+\|Y_t\|_{L^{2p}(I)}^{2p}\:dt.
\end{equation}

\noindent
Clearly, for some constant $C>0$ we have
$$
\E \Big(\|W_A(t)\|_{L^8(I)}^8\Big)\leq C
\left(\sum\limits_{k=1}^{+\infty}\frac{1}{\lambda_k^{1+2 \gamma_{0}}}\right)^4<+\infty,
$$
so that \eqref{est3} implies that
\begin{equation}
\int_0^T \E\left(\Big \|\Big(\frac{dY_t}{dx}\Big)\cdot
Y_t^{p-1}\Big\|^2\right)\:dt\leq
C_1\E\Big(\|Y_0\|_{L^{2p}(I)}^{2p}\Big)+\tilde{C_2}+C_2\int_0^T\E\Big(\|Y_t\|_{L^{2p}(I)}^{2p}\Big)\:dt
\end{equation}
for uniform constants $C_1$, $\tilde{C_2}$ and $C_2$. Next, observe
that for $\beta < \frac 14+\gamma_0$ we have
\begin{equation*}
\sup\limits_{t\geq
0}\E\left(\|W_A(t)\|^{2p}_{\beta}\right)<+\infty,\quad p\geq 1,
\end{equation*}
hence for $T\geq 0$ we have
\begin{equation*}
\begin{split}
\int_0^T \E\Big(\|X_t^3\|_{\beta}^2\Big)\:dt &\leq C\left(\int_0^T
\E\Big(\|Y_t^3\|_{\beta}^2\Big)\:dt+\int_0^T
\E\Big(\|W_A(t)^3\|_{\beta}^2\Big)\:dt\right)\\ &\leq
C\left(\int_0^T \E\Big(\|Y_t^3\|_{\frac 12}^2\Big)\:dt+\int_0^T
\E\Big(\|W_A(t)^3\|_{\beta}^2\Big)\:dt\right)\\
&\leq C\left(T+\E\Big(\|Y_0\|_{L^6(I)}^6\Big)+\int_0^T
\E\Big(\|Y_t(x)\|^6_{L^6(I)}\Big)\:dt\right)\\
&\leq C\left(T+\E\Big(\|X_0\|_{L^6(I)}^6\Big)+\int_0^T
\E\Big(\|X_t(x)\|^6_{L^6(I)}\Big)\:dt\right)\\
&= C\left(T+\E\Big(\|X_0\|_{L^6(I)}^6\Big)+\int_0^T
\E\Big(\|X_t(x)^3\|^2\Big)\:dt\right)\, ,
\end{split}
\end{equation*}
where the constant $C$ may change from line to line. Thus integrating with respect
to $\mu$ and using the fact that $(X_t)_{t\geq 0}$ is a stationary solution for
\eqref{SB} with invariant distribution $\mu$ we get
$$
T\int_H \|x^3\|^2_{\beta}\:\mu(dx)\leq \tilde{C}\left(T+\int_H
\|x\|_{L^6(I)}^6\:\mu(dx)+T\int_H \|x^3\|^2\:\mu(dx)\right).
$$
This yields the statement of the lemma.
\end{proof}

\medskip
\noindent
If we denote by $L$ the Kolomogorov operator associated
with \eqref{SB}, we have by Theorem \ref{uniqueness} the closure
$(\bar{L}, D(\bar{L}))$ of $L$ in $L^{1}(H,\mu)$ generates a
$C_0$-semigroup of contractions $(\bar{P}_t)_{t\geq0}$ on
$L^1(H,\mu)$, $(\bar{P}_t)_{t\geq0}$ is Markovian and the measure
$\mu$ is $(\bar{P}_t)_{t\geq0}$-invariant.

\subsection{Stochastic equations modeling thin-film growth}
Let us consider the following stochastic partial differential
equation
\begin{equation}
\label{BH} du(t,x) = \left( -\partial_x^{(4)}u(t,x) +
\nu\partial_x^{(2)}u(t,x)-\partial_x^{(2)}\Big(\partial_x
u(t,x)\Big)^2\right)\, dt + \eta (t,x) \, , \quad (t,x) \in
\R_+\times I\, ,
\end{equation}
where $\nu\geq 0$, $\eta(t,x)=dW_t(x)$ is a Wiener process on
$L^2([0,1])$ with covariance operator satisfying
$$
Qe_k=q_k e_k,\quad k\geq 1,\quad (q_k)_{k\geq 1}\subseteq
\ell^{\infty}(\N),
$$
and $(e_k)_{k\geq 1}$ denotes the orthonormal basis of $H\defeq
L^2_0([0,1])=\{f\in L^2([0,1]):\:\:\int_0^1 f(r)\:dr=0\}$ consisting
of eigenvectors of the self-adjoint extension of
$$
A\defeq -\partial_x^{(4)}u+\nu \partial_x^{(2)}u
$$
in $H$ with periodic or Neumann boundary conditions.

\noindent
Bl\"omker and Hairer proved in \cite{BH} the existence of a
stationary solution of \eqref{BH}. In particular, they showed the
existence of an invariant measure $\mu$ satisfying the moment
estimate
$$
\int_H \log (1+\|x\|^2)\:\mu(dx)<+\infty.
$$
The purpose of the example is to demonstrate how to obtain improved
a priori moment estimates on $\mu$, using Theorem \ref{mainresult}.
To this end note that the coefficients of \eqref{BH} satisfy the
inequality
\begin{equation}\label{LTF}
\la Ay+B(y+w),y\ra \leq -\frac 14 \|y\|^2_{\frac
12}+\beta\|w\|^8_{\frac {5}{16}}\|y\|^2+\gamma\|w\|^2_{\frac{5}{16}}
\end{equation}
for suitable constants $\beta$ and $\gamma$, since
\begin{equation}\label{T-est1}
\begin{split}
\int_0^1 \partial_x^{(2)}\Big(
\partial_x(y+w)\Big)^2(x)y(x)\:dx&=2\int_0^1
\partial_x^{(2)}y(x)\cdot \partial_x y(x)\cdot \partial_x
w(x)\:dx+\int_0^1 \partial_x^{(2)}y(x) \Big( \partial w(x)
\Big)^2\:dx\\
&\hskip-2em\leq \frac 14 \|y\|^2_{\frac 12}+\int_0^1 \Big(\partial_x
y(x)\Big)^2 \Big(\partial_x w(x)\Big)^2\:dx+\frac 14\|y\|^2_{\frac
14}+\|\partial_x w\|^4_{L^4([0,1])}\\
&\hskip-2em\leq \frac 12 \|y\|^2_{\frac 12}+ C \|\partial_x
w\|^2_{L^2([0,1])}\|\partial_x y\|^2_{\infty}+\|\partial_x
w\|^4_{L^4([0,1])}.
\end{split}
\end{equation}

\noindent
Using the fact that $W^{\frac{5}{4},2}(0,1)\subset W^{1,4}(0,1)$
and $\|u\|_{W^{\frac{5}{4},2}(0,1)}\leq C\|w\|_{\frac{5}{16}}$, we
conclude that
\begin{equation}
\label{T-est2}
\int_0^1 \partial_x^{(2)}\Big(\partial_x(y+w)\Big)^2(x)y(x)\:dx
\leq \frac 34 \|y\|^2_{\frac 12}+ \beta\|\partial_x w\|^8_{\frac{5}{16}}\|y\|^2
+ \gamma\|w\|^2_{\frac{5}{16}}
\end{equation}
for suitable constants $\beta$, $\gamma$, hence \eqref{LTF} follows.

\noindent
We can arrange the eigenvalues $(\lambda_k)_{k\geq 1}$ of
$-A$ in such a way that $\lambda_k=4\pi^2k^2(4\pi^2k^2+\nu )$ with
multiplicity $1$ (in the case of Neumann boundary conditions) or $2$
(in the case of periodic boundary conditions). Then ${\bf (H_4)}$ is
satisfied for $\kappa\defeq \delta-\gamma-\ve>\frac 18$, $\delta\in
(0,\frac 12)$, because then
\begin{equation*}
Z_{\gamma,\delta,\ve}
\leq 2\sum\limits_{k=1}^{+\infty}\Big(4\pi^2k^2(4\pi^2k^2+\nu) \Big)^{-2(\delta-\gamma-\ve)}\cdot
\|q\|_{\infty}
\leq C\sum\limits_{k=1}^{+\infty}\kappa^{-8 \kappa}\quad \mbox{if
$\kappa>\frac 18$}.
\end{equation*}

\noindent
Theorem \ref{mainresult} now implies the following improved a priori moment estimates
\begin{cor}
For any invariant measure $\mu$ of \eqref{BH} we have

\begin{enumerate}
\item [(i)] $\int \|x\|_0^{2p}\:\mu(dx)<\infty$ for $p\ge 0$.
\item [(ii)] $\int \|x\|^2_\sigma\|x\|_0^{2p}\:\mu(dx) < \infty$ for $p\ge 0$, $\sigma\leq\frac 12$.
\end{enumerate}

\end{cor}

\end{document}